\documentclass[11pt,thmsa]{article}%
\usepackage{amsmath}
\usepackage{amsfonts}
\usepackage{amssymb}
\usepackage{graphicx}%
\newtheorem{theorem}{Theorem}[section]

\newtheorem{corollary}[theorem]{Corollary}

\newtheorem{definition}[theorem]{Definition}

\newtheorem{lemma}[theorem]{Lemma}

\newtheorem{question}[theorem]{Question}
\newtheorem{proposition}[theorem]{Proposition}

\newenvironment{proof}[1][Proof]{\noindent\textbf{#1.} }{\ \rule{0.5em}{0.5em}}

\begin{document}
\title{$\delta$-homogeneity in Finsler geometry and the positive curvature problem\thanks{Supported by NSFC (no. 11271216)}}
\author{Ming Xu$^1$ and Lei Zhang$^2$ \thanks{Corresponding author. E-mail: 546502871@qq.com}\\
\\
$^1$College of Mathematics\\
Tianjin Normal University\\
 Tianjin 300387, P.R. China\\
 \\
$^2$School of Mathematical Sciences\\
Nankai University\\
Tianjin 300071, P.R. China}
\date{}
\maketitle

\begin{abstract}
In this paper, we explore the similarity between normal homogeneity and $\delta$-homogeneity in Finsler geometry. They are both non-negatively curved Finsler spaces. We show that any connected $\delta$-homogeneous Finsler space is $G$-$\delta$-homo-geneous, for some suitably chosen connected quasi-compact $G$. So $\delta$-homogeneous Finsler metrics can be defined by a bi-invariant singular metric on $G$ and submersion, just as normal
homogeneous metrics, using a bi-invariant Finsler metric on $G$ instead. More careful analysis shows, in the space of all Finsler metrics on $G/H$, the subset of all
$G$-$\delta$-homogeneous ones is in fact the closure for the subset of all $G$-normal ones, in the local $C^0$-topology (Theorem \ref{main-thm-1}). Using this approximation technique, the classification work for
positively curved normal homogeneous Finsler spaces can be applied to classify positively curved $\delta$-homogeneous Finsler spaces, which provides the same classification list. As a by-product, this argument
tells more about $\delta$-homogeneous Finsler metrics satisfying the
(FP) condition (a weaker version of positively curved condition).



\end{abstract}

\section{Introduction}

The concept of $\delta$-homogeneity is introduced by V. N. Berestovskii and C. P. Plaut \cite{BP1999}, and extensively studied
in Riemannian geometry \cite{BN2008} \cite{BN2009} \cite{BN2014}  \cite{BN2011}. Recall that a connected metric space $(M,d)$ is called $G$-$\delta$-homogeneous, where $G$ is a closed connected subgroup of the isometry group $I(M,d)$, if for any point $x,y\in M$,
there exists an isometry $g\in G$ satisfying $g(x)=y$ and the displacement function
$f(\cdot)=d(\cdot,g(\cdot))$ on $M$ reaches its maximum at $x$ (such an isometry $g\in G$ is called a $\delta(x)$-translation).

In \cite{ZD2016}, L. Zhang and S. Deng defined and studied $\delta$-homogeneity in Finsler geometry. They provided several equivalent
descriptions for the $G$-$\delta$-homogeneity. Instead of the isometry $g\in G$ which achieves the transitivity and maximal displacement, we can also use
$\delta(\cdot)$-Killing vector fields or $\delta$-vectors from $\mathfrak{g}=\mathrm{Lie}(G)$ (see \cite{ZD2016} or
Section 2). Their alternative description for a $G$-$\delta$-homogeneous Finsler metric $F$ on $G/H$,
is the foundation for defining the Chebyshev metric $\tilde{F}$, i.e. a bi-invariant singular metric on $G$ induced by $F$ on $M=G/H$. Then $F$ is canonically determined by the Chebyshev metric and the Finsler submersion $\pi:(G,\tilde{F})\rightarrow (G/H,F)$. This alternative description for $\delta$-homogeneity in Finsler geometry is very similar to that for normal homogeneity, which uses smooth bi-invariant Finsler metrics on $G$ instead. Though defining the Chebyshev metric, or any other bi-invariant smooth or singular metric on $G$, will require $G$ to be quasi-compact, i.e. $\mathfrak{g}$ is compact. We prove this is not an essential obstacle, i.e. for any connected $\delta$-homogeneous Finsler space $(M,F)$, we can choose a suitable connected quasi-compact Lie group $G$, such that $(M,F)$ is $G$-$\delta$-homogeneous (see Theorem \ref{quasi-compactness-thm} below).

So both normal homogeneity and $\delta$-homogeneity can be defined by  submersion. This explains the phenomenon that they share many geometric
properties, for example, both have 0 S-curvature and non-negative flag curvature. This similarity in their defining patterns can be formulated as the following approximation theorem,
where we applies the fundamental technique of approximating a singular metric (or a singular norm) by smooth ones in the local $C^0$-topology,

\begin{theorem} \label{main-thm-1}
Assume $G$ is a connected quasi-compact Lie group which acts effectively
on the smooth coset space $G/H$. In the space of all smooth Finsler
metric on $G/H$,
the subset of all $G$-$\delta$-homogeneous ones coincides with the closure for
all $G$-normal homogeneous ones
in the local $C^0$-topology.
\end{theorem}

Here the local $C^0$-topology means, on the fixed manifold $M$, the sequence of singular or smooth Finsler metrics $\{F_n:TM\rightarrow \mathbb{R}\}$ converge to $F:TM\rightarrow \mathbb{R}$, iff they uniformly converge to $F$ when restricted each compact subset of $TM$.

Theorem \ref{main-thm-1} brings a natural question: what geometric properties of $G$-normal homogeneous Finsler metrics can be passed to their local $C^0$-limit metrics, i.e. those $G$-$\delta$-homogeneous ones?
Studying this question in the general context seems to be very
intriguing. But local
$C^0$-convergence is weak for studying most curvature properties, and
there is no convenient comparison theorems in Finsler geometry as in Riemannian geometry.

In this paper, we will use Theorem \ref{main-thm-1}, and the method in \cite{XD2016} to study
the classification of positively curved $\delta$-homogeneous Finsler spaces.
The key observation in \cite{XD2016} is that a flat splitting subalgebra (see \cite{XD2016} or Section 4) provides totally geodesic flat subspace with a dimension bigger than 1 in a normal homogeneous Finsler space. So it is an obstacle for positive flag curvature.
It still works for $\delta$-homogeneous Finsler spaces (see Lemma \ref{lemma-4}), so we have the following theorem.

\begin{theorem} \label{main-thm-2}
Let $G$ be a compact connected Lie group which acts effectively on
the smooth coset space $G/H$. Then $G/H$ admits positively curved $G$-$\delta$-homogeneous
Finsler metrics iff it admits positively curved $G$-normal homogeneous Riemannian metrics.
\end{theorem}

Notice a positively curved homogeneous Finsler space $(M,F)$ is compact by Bonnet-Myers Theorem, so does its isometry group $I(M,F)$, through which $G$ acts on $G/H$. The classification for positively curved homogeneous
spaces is only up to local isometries. We can "cancel" some local product
factor of $G$ contained in $H$ without changing the metric, and at the same
time with the effectiveness of $G$ satisfied.
So Theorem \ref{main-thm-2} indicates the classification list for positively
curved $\delta$-homogeneous Finsler spaces coincides with that in \cite{XD2016}, or that of M. Berger in \cite{Ber} (plus B. Wilking's space
\cite{Wi}). To be precise, it consists of
\begin{description}
\item{\rm (1)} Rank one compact symmetric spaces
\begin{eqnarray*}
& &S^{n-1}=\mathrm{SO}(n)/\mathrm{SO}(n-1),
\mathbb{C}\mathrm{P}^{n-1}=\mathrm{SU}(n)/\mathrm{SU}(n-1),\\
& &\mathbb{H}\mathrm{P}^{n-1}=\mathrm{Sp}(n)/\mathrm{Sp}(n-1)\mathrm{Sp}(1), \mbox{ and }
\mathbb{F}_4/\mathrm{Spin}(9).
\end{eqnarray*}
\item{\rm (2)} Other homogeneous spheres and complex projective spaces
\begin{eqnarray*}
& &S^{2n-1}=\mathrm{SU}(n)/\mathrm{SU}(n-1)=\mathrm{U}(n)/\mathrm{U}(n-1),\\
& &S^{4n-1}=\mathrm{Sp}(n)/\mathrm{Sp}(n-1)=\mathrm{Sp}(n)\mathrm{U}(1)
/\mathrm{Sp}(n-1)\mathrm{U}(1)\\
& &\phantom{S^{4n-1}}
=\mathrm{Sp}(n)\mathrm{Sp}(1)/\mathrm{Sp}(n-1)\mathrm{Sp}(1),\\
& &S^6=\mathrm{G}_2/\mathrm{SU}(3), S^7=\mathrm{Spin}(7)/\mathrm{G}_2,\\
& &S^{15}=\mathrm{Spin}(9)/\mathrm{Spin}(7), \mbox{ and }
\mathbb{C}\mathrm{P}^{2n-1}=\mathrm{Sp}(n)/\mathrm{Sp}(n-1)\mathrm{U}(1).
\end{eqnarray*}
\item{\rm (3)} Berger's spaces $\mathrm{SU}(5)/\mathrm{Sp}(2)S^1$ and $\mathrm{Sp}(2)/\mathrm{SU}(2)$.
\item{\rm (4)} Wilking's space $S^{1,1}=\mathrm{SU}(3)\times\mathrm{SO}(3)/\mathrm{U}(2)$.
\end{description}

The flat splitting subalgebra and the totally geodesic technique can tell more
when we study $\delta$-homogeneous Finsler spaces satisfying the (FP) condition (a weaker version of positive flag curvature condition, which is also called the flag-wise positively curved condition; see \cite{XD2016-2} or Section 4). We prove a flag-wise positively curved $\delta$-homogeneous Finsler space $G/H$ must be compact, and satisfy the rank inequality as positively curved homogeneous spaces \cite{XDHH2016}. When $\dim G/H$ is even,
it admits flag-wise positively curved $\delta$-homogeneous (or normal homogeneous) Finsler metrics iff it admits positively curved $\delta$-homogeneous (or normal homogeneous) Finsler metrics. But when $\dim G/H$ is odd, there are many more examples of flag-wise positively curved $\delta$-homogeneous Finsler spaces than those positively curved ones.

\section{Preliminaries}

In this section, we collect some fundamental knowledge on Finsler geometry from \cite{BCS} and \cite{Sh2001}, and Finsler submersion from \cite{PD}.

\subsection{Minkowski norm and Finsler metric}
A {\it Minkowski norm} on a real vector space $\mathbf{V}$, $\dim \mathbf{V}=n$,
is a continuous function $F:\mathbf{V}\rightarrow [0,+\infty)$ satisfying the
following conditions:
\begin{description}
\item{\rm (1)} Positiveness and smoothness: $F$ is a positive smooth function on $\mathbf{V}\backslash \{0\}$.
\item{\rm (2)} Positively homogeneity of degree one: $F(\lambda y)=\lambda F(y)$ when
$\lambda\geq 0$.
\item{\rm (3)} Strong convexity: given any basis $\{e_1,\ldots,e_n\}$ of
$\mathbf{V}$ and
correspondingly the linear coordinates $y=y^i e_i$,
the Hessian matrix
\begin{equation*}
(g_{ij}(y))=\left(\frac12[F^2(y)]_{y^iy^j}\right)
\end{equation*}
is positive definite whenever $y\neq 0$.
\end{description}
We will call $(\mathbf{V},F)$ a {\it Minkowski space}.

The Hessian matrix defines an inner product
$\langle\cdot,\cdot\rangle_y^F$ on $\mathbf{V}$ by
$$\langle u,v\rangle_y^F=g_{ij}(y)u^i v^j=\frac12
\frac{\mathrm{d}^2}{\mathrm{d}t\mathrm{d}s}
F^2(y+tu+sv)|_{s=t=0},$$
from which we see it
is independent of the choice for the linear basis.

A {\it Finsler metric} on a smooth manifold $M$ is a continuous
function $F$ on $TM$ such that its restriction to the slit tangent bundle
$TM\backslash0$ is smooth, and its restriction to each tangent space
is a Minkowski norm. We will also call $(M,F)$ a {\it Finsler space}. Given any smooth tangent field $Y$ on $M$ which is non-vanishing
everywhere in an open set $\mathcal{U}\subset M$, the Hessian matrices $(g_{ij}(x,Y(x)))$ or equivalently the inner
products $\langle\cdot,\cdot\rangle^F_{Y(x)}$ at each $x\in\mathcal{U}$
define a smooth Riemannian metric on $\mathcal{U}$. We will simply call this metric the {\it localization of $F$ at $Y$}, and denote it as $g_Y^F$.

Important examples of Finsler metrics include
Riemannian metrics, Randers metrics, $(\alpha,\beta)$-metrics, etc.
Riemannian metrics are a special class of Finsler metrics whose
 Hessian matrices at each point only depends on $x\in M$ rather than $y\in T_x M$.
A Riemannian  metric can also be defined as a  global smooth section
$g_{ij}dx^i dx^j$ of $\mathrm{Sym}^2(T^* M)$.
Randers metrics are the most simple and important class of non-Riemannian metrics in Finsler geometry. They
are defined by $F=\alpha+\beta$, in which $\alpha$ is a Riemannian metric and $\beta$ is
a 1-form. They can be naturally generalized as $(\alpha,\beta)$-metrics which have
the form $F=\alpha\phi(\beta/\alpha)$ with a positive smooth function $\phi$ and similar
$\alpha$ and $\beta$ as Randers metrics. In recent years, there have been a lot of research works for $(\alpha,\beta)$-metrics as well as for Randers metrics.

\subsection{Geodesic, Riemannian curvature, and totally geodesic subspace}
On a Finsler space $(M,F)$, we usually choose the {\it standard local coordinates} $(x^i,y^j)$, where $x=(x^i)\in M$
and $y=y^j\partial_{x^j}\in T_x M$, to present the connections, curvatures, and other geometric quantities.

The geodesics are important geometric subjects in Finsler geometry. They are smooth curves which satisfies the local minimizing principle with
respect to the distance function $d_F(\cdot,\cdot)$ induced by the Finsler metric $F$. Notice $d_F(\cdot,\cdot)$ may not be reversible in general.
We will always parametrize a geodesic $c(t)$ on $(M,F)$ to have a nonzero constant speed, i.e. the $F$-length of the tangent field $\dot{c}(t)=\frac{d}{dt}c(t)$ is a positive constant for all $t$. Then the geodesics
can be equivalently defined as following.

First, we have a globally defined smooth vector field $\mathbf{G}$ on $TM\backslash 0$, called the
{\it the geodesic spray}. For any standard
local coordinates, it can be presented as
\begin{equation}
\mathbf{G}=y^i\partial_{x^i}-2\mathbf{G}^i\partial_{y^i},
\end{equation}
where
\begin{equation}
\mathbf{G}^i=\frac{1}{4} g^{il}([F^2]_{x^k y^l}y^k - [F^2]_{x^l}).
\end{equation}
Then a curve $c(t)$ on $M$ is a geodesic of nonzero constant speed if and only if $(c(t),\dot{c}(t))$ is an integration curve
of $\mathbf{G}$. For standard local coordinates, a geodesic $c(t)=(c^i(t))$ satisfies the equations
\begin{equation}
\ddot{c}^i(t)+2 \mathbf{G}^i(c(t),\dot{c}(t))=0.
\end{equation}

The coefficients $\mathbf{G}^i$ of the geodesic spray $\mathbf{G}$ are important for us to present curvatures in Finsler geometry.

For example, we have a similar curvature as in the Riemannian case, which
is called the {\it Riemann curvature}. It can be defined either by the Jacobi field equation for the variation of a constant speed geodesic,
or by the structure equation for the curvature of the Chern connection.

Using any standard local coordinates,
the Riemannian curvature can be presented as
$R^F_y=R^i_k(y)\partial_{x^i}\otimes dx^k:T_x M\rightarrow T_x M$, where
\begin{equation}
R^i_k(y)=2\partial_{x^k}G^i-y^j\partial^2_{x^j y^k}G^i+2G^j\partial^2_{y^j y^k}G^i
-\partial_{y^j}G^i\partial_{y^k}G^j.
\end{equation}

Using the Riemannian curvature, the sectional curvature can be generalized to Finsler
geometry, which is called the {\it flag curvature}. We call  $(x,y,\mathbf{P})$ a {\it flag triple}, if $x$ is a point in $M$, $\mathbf{P}$ is a tangent plane in some $T_xM$, and $y$ is a nonzero vector in $\mathbf{P}$.
Assume $\mathbf{P}$ is linearly spanned by $y$ and $v$, then the flag curvature
for $(x,y,\mathbf{P})$ is
\begin{equation}
K^F(x,y,\mathbf{P})=\frac{\langle R_y v,v\rangle^F_y}
{\langle y,y\rangle^F_y \langle v,v\rangle^F_y-(\langle y,v\rangle^F_y)^2}.
\end{equation}
When $F$ is a Riemannian metric, it is just the sectional curvature for $(x,\mathbf{P})$, which is independent of the choice of $y$.
ure
Using Riemann curvature or flag curvature, the Ricci curvature can also
be generalized to Finsler geometry, i.e. for any $x\in M$ and nonzero tangent
vector $y\in T_xM$,
$$\mathrm{Ric}^F(x,y)=\mathrm{tr}(R^F_y)=\sum_{i=1}^n F(y)^2 K^F(x,y,y\wedge v_i),$$
where $\{v_1,\ldots,v_n\}$ is an orthogonal basis for the $g_y^F$-orthogonal
complement of $y$ in $T_xM$.

Z. Shen has an important observation that all these curvatures described above can
be closely related to Riemannian geometry (see Proposition 6.2.2 in \cite{Sh2001}), so we may call them {\it Riemannian curvatures}.
In this work, we will need the following refinement of his observation.

\begin{proposition}\label{prop-1}
Let $Y$ be a smooth vector field on a Finsler space $(M,F)$, such that
$y=Y(x)\neq 0$, and $Y$ generates a geodesic of constant speed through $x$,
then $R^F_y=R^{g_Y^F}_y$, $\mathrm{Ric}^F(x,y)=\mathrm{Ric}^{g^F_Y}(x,y)$, and for any tangent plane $\mathbf{P}$
in $T_xM$ containing $y$, $K^F(x,y,\mathbf{P})=K^{g^F_Y}(x,\mathbf{P})$.
\end{proposition}

The statements for $R^F_y$ and $K^F(x,y,\mathbf{P})$ are contained in Theorem 4.2 in \cite{XDHH2016}. The statement for $\mathrm{Ric}^F(x,y)$ then follows
easily.

A submanifold $N$ of a smooth Finsler space $(M,F)$ can be naturally endowed with a smooth submanifold metric, denoted as $F|_N$. At each $x\in N$, the Minkowski norm $F|_N(x,\cdot)$ is just the restriction of the Minkowski norm $F(x,\cdot)$ from $T_xM$ to $T_xN$. We say that $(N,F|_N)$ is a {\it Finsler
submanifold} or a {\it Finsler subspace}.

For the study of Riemann curvature and flag curvature, the most important
Finsler subspaces are totally geodesic subspaces. A Finsler subspace $(N,F|_N)$ is totally geodesic iff, for any standard local
coordinate system $(x^i,y^j)$ such that $N$ is locally defined by
$x^{k+1}=\cdots=x^n=0$ where $n=\dim M$, we have
$$G^i(x,y)=0,\quad k<i\leq n, y\in T_xN.$$
A direct calculation shows that in this case, the Riemann curvature
$R_y^{F|_N}:T_xN\rightarrow T_xN$ of $(N,F|_N)$ is just the restriction
of the Riemann curvature $R^F_y$ of $(M,F)$, when $y$ is a nonzero
tangent vector in $T_xN$ at $x\in N$. Therefore, we have \cite{XD2016}
\begin{proposition}
Let $(N,F|_N)$ be a totally geodesic submanifold of $(M,F)$. Then for
any flag $(x,y,\mathbf{P})$ in $N$ (i.e. $x\in N$, $y$ is nonzero tangent
vector in $T_xN$, and $\mathbf{P}$ is a tangent plane containing $y$),
we have $$K^{F|_N}(x,y,\mathbf{P})=K^F(x,y,\mathbf{P}).$$
\end{proposition}

\subsection{Finsler Submersion}

A linear map $\pi: (\mathbf{V}_1,F_1)\rightarrow (\mathbf{V}_2,F_2)$ between two Minkowski
spaces is called a {\it Finsler submersion}, if we have
\begin{equation}\label{0001}
\pi(\{w\in \mathbf{V}\mbox{ with }F_1(w)\leq 1\})=
\{u\in \mathbf{V}\mbox{ with }F_2(u)\leq 1\}.
\end{equation}
It is obvious that a submersion map
$\pi$ must be surjective, and that the Minkowski norm $F_2$ on $\mathbf{V}_2$ is uniquely determined by the following equality,
$$F_2(u)=\inf\{F_1(w)|\pi(w)=u\}.$$
Given the Minkowski space $(\mathbf{V}_1,F_1)$ and a surjective linear map
$\pi:\mathbf{V}_1\rightarrow\mathbf{V}_2$, there exists a unique Minkowski norm
$F_2$ on $\mathbf{V}_2$ such that $\pi$ is a submersion. We usually say that $F_2$ is induced by $F_1$ and submersion.

Given the submersion $\pi:(\mathbf{V}_1,F_1)\rightarrow
(\mathbf{V}_2,F_2)$, the horizonal lift of a nonzero vector $u\in\mathbf{V}_2$
is the unique $w\in\mathbf{V}_1$ satisfying the conditions
\begin{equation}\label{horizonal-lift}
\pi(w)=u\mbox{ and }F_1(w)=F_2(u).
\end{equation}
Then
$\pi:(\mathbf{V}_1,\langle\cdot,\cdot\rangle_w^{F_1})\rightarrow (\mathbf{V}_2,\langle\cdot,\cdot\rangle_u^{F_2})$ is
also a submersion between Euclidean spaces.


A smooth map $\pi:(M_1,F_1)\rightarrow (M_2,F_2)$ between two Finsler spaces
is called a submersion, if for any $x\in M_1$ the induced tangent map
$\pi_*:(T_x M_1,F_1(x,\cdot))\rightarrow (T_{\rho(x)}M_2,F_2(\rho(x),\cdot))$ is
a Finsler submersion between Minkowski spaces.
Given the surjective
smooth map $\pi:M_1\rightarrow M_2$, and a Finsler metric $F_1$ on $M_1$, if
there exists a metric $F_2$ on $M_2$ which makes $\pi$ a Finsler submersion, we will call $F_2$
the {\it induced metric by $F_1$ and the submersion $\pi$}. The induced metric must be unique, even though it usually does not exist.

For a Finsler submersion $\pi:(M_1,F_1)\rightarrow (M_2,F_2)$, we can define the {\it horizonal lift
for flag triples}. We call the flag triple $(x_1,y_1,\mathbf{P}_1)$ on $M_1$ the horizonal lift of the flag triple
$(x_2,y_2,\mathbf{P}_2)$ on $M_2$, iff $\pi(x_1)=x_2$, $y_2$ is the horizonal
lift of $y_1$, and $\mathbf{P}_1$ is the horizonal lift of $\mathbf{P}_2$
with respect to $\pi_*:(T_{x_1}M_1,\langle\cdot,\cdot\rangle_{y_1}^{F_1})
\rightarrow(T_{x_2}M_2,\langle\cdot,\cdot\rangle_{y_2}^{F_2})$.

The importance of Finsler submersion for the study of flag curvature is implied by the following theorem in \cite{PD}.
\begin{theorem}\label{theorem-3-2}
Let $\pi:(M_1,F_1)\rightarrow (M_2,F_2)$ be a Finsler submersion, and
the flag $(x_1,y_1,\mathbf{P}_1)$ on $M_1$ is the horizonal lift for
the flag $(x_2,y_2,\mathbf{P}_2)$ on $M_2$. Then we have
\begin{equation}
{K}^{F_1}(x_1,y_1,y_1\wedge v_1)\leq
K^{F_2}(x_2,y_2,y_2\wedge v_2).
\end{equation}
\end{theorem}

\section{$\delta$-homogeneous Finsler metrics}

\subsection{Singular Minkowski norms and singular metrics}

In this work, we also need to consider singular norms and singular metrics. If not specified, Minkowski norms and Finsler metrics are
referred to the smooth ones defined in Subsection 2.1. Notice in some
literatures, for the notions of norms and metrics, people use "continuous" instead of "singular", "regular"
instead of "smooth".

A non-negative continuous function $F$ on a real vector space $\mathbf{V}$ is called a singular norm, if it satisfies the following conditions:
\begin{description}
\item{\rm (1)} Positiveness: $F$ is a positive function on $\mathbf{V}\backslash\{0\}$.
\item{\rm (2)} Positive homogeneity of degree one:
$F(\lambda y)=\lambda F(y)$ when $\lambda\geq 0$.
\item{\rm (3)} Convexity: $F(\lambda y_1+(1-\lambda) y_2)\leq \lambda F(y_1)+(1-\lambda)F(y_2)$ whenever $\lambda\in[0,1]$.
\end{description}
We will also call $(\mathbf{V},F)$ a {\it singular norm space}.

A geometric way to describe the convexity condition for a singular norm $F$ on $\mathbf{V}$ is provided by its indicatrix
$\mathcal{I}_{F,\mathbf{V}}=\{y\in\mathbf{V}|F(y)=1\}$. The indicatrix $\mathcal{I}_{F,\mathbf{V}}$ for a singular norm $F$ on $\mathbf{V}$
is a sphere surrounding the origin $o\in\mathbf{V}$,
which bounds a convex region in $\mathbf{V}$. The indicatrix $\mathcal{I}_{F,\mathbf{V}}$ is a smooth sub-manifold iff
$F$ is smooth on $\mathbf{V}\backslash\{0\}$. Notice even in this situation, $F$ has not been guaranteed to be a Minkowski norm, because we only have the semi
positive definiteness for the Hessian $(g_{ij}(y))=(\frac12[F^2(y)]_{y^iy^j})$, rather than the positive definiteness for the strong convexity condition.

Similarly, a {\it singular metric} $F$ on a smooth manifold $M$ is a continuous function $F$ on $TM$, such that it is positive on the slit tangent bundle $TM\backslash 0$, and its restriction to each tangent space is a singular norm. We will call $(M,F)$ a {\it singular metric space}.
In practice, for the homogeneous singular metrics we will consider in this
work, the smoothness is broken in the tangent direction (i.e. within $T_xM$ for each $x\in M$), but still kept in some sense along the manifold directions.

Geodesics of singular metric spaces can still be defined using locally
minimizing principle. But it is not hard to find examples that there
exists more than one minimizing geodesic from $x$ to $y$ within a neighborhood
$\mathcal{U}$ of $x$, no matter how small $\mathcal{U}$ is. On the other hand,
most curvature concepts in Finsler geometry can not be easily generalized
to the singular situation.

Submersions for singular norm spaces and singular metric spaces
can be similarly defined as in Subsection 2.3. We can still use (\ref{horizonal-lift}) to define the horizonal lift of a tangent vector. But the horizonal lift may not be unique. Further more,
we do not have a flag curvature inequality for Finsler submersion as in Theorem \ref{theorem-3-2}.

\subsection{Definition and Properties}

In \cite{ZD2016}, connected $\delta$-homogeneous Finsler spaces (we only consider connected Lie groups and connected Finsler spaces in this work) are defined as following.

\begin{definition} Let $G$ be a connected Lie group and $(G/H,F)$ a
$G$-homogeneous Finsler space. We call $(G/H,F)$ a $G$-$\delta$-homogeneous Finsler space, if
for any $x,y\in G/H$, there exists an element $g\in G$, such that $g(x)=y$ and
the displacement function $f(\cdot)=d_F(\cdot,g(\cdot))$ of $g$ reaches its maximum at $x$. More generally, a connected Finsler space $(M,F)$ is called
$\delta$-homogeneous if it is $G$-$\delta$-homogeneous for $G=I_0(M,F)$.
\end{definition}

We may always assume $G$ is a closed connected subgroups of the connected isometry group $I_0(G/H,F)$, i.e. $G$ acts effectively on $G/H$.
We apply the following fundamental algebraic setup.
Choose an $\mathrm{Ad}(H)$-invariant decomposition  $\mathfrak{g}=\mathfrak{h}+\mathfrak{m}$,
where $\mathfrak{h}=\mathrm{Lie}(H)$ and $\mathfrak{m}$ can be identified
with the tangent space $T_o(G/H)$ at $o=eH$. We will denote
$\mathrm{pr}_\mathfrak{h}$ and $\mathrm{pr}_\mathfrak{m}$ the corresponding projections to
$\mathfrak{h}$ and $\mathfrak{m}$ respectively.
In this setup, a $G$-homogeneous
Finsler metric $F$ on $G/H$ is one-to-one determined by an
$\mathrm{Ad}(H)$-invariant Minkowski norm on $\mathfrak{m}$, which for simplicity,
will still be denoted as $F$ \cite{D}. This setup is also correct when $F$ is singular.

S. Deng and L. Zhang gave an equivalent description for $\delta$-homogeneous
Finsler spaces in \cite{ZD2016}, i.e. the following theorem,
\begin{theorem}\label{thm-1}
Let $G$ be a connected Lie group and $(G/H,F)$ a
$G$-homogeneous Finsler space.
Then $(G/H,F)$ is $G$-$\delta$-homogeneous iff one of two equivalent
conditions is satisfied:
\begin{description}
\item{\rm (1)}
For any $u\in\mathfrak{m}$, we can find a $\delta$-vector $\tilde{u}\in\mathfrak{g}$ for $u$, i.e. $\mathrm{pr}_\mathfrak{m}(\tilde{u})=u$, and the function
$f(\cdot)=F(\mathrm{pr}_\mathfrak{m}(\mathrm{Ad}(\cdot)\tilde{u}))$
achieves its maximum at $e$.
\item{\rm (2)}
For any $x\in G/H$ and any tangent vector $u\in T_x(G/H)$,  we can find
a $\delta(x)$-Killing vector field $X$ from $\mathfrak{g}$ for the tangent
vector $u$, i.e. $X(x)=u$, and the function $f(\cdot)=F(X(\cdot))$ achieves
its maximum at $x$.
\end{description}
\end{theorem}

The $\delta$-vectors in (1) provides the $\delta(o)$-Killing vectors in (2), and their $\mathrm{Ad}(G)$-orbits provides the $\delta(\cdot)$-Killing vectors for all points and all tangent directions.

Here are some fundamental properties of $\delta$-homogeneous Finsler spaces.

\begin{theorem} \label{non-nega-thm}
A $\delta$-homogeneous Finsler space $(M,F)$ has non-negative
flag curvature.
\end{theorem}

\begin{proof}
By Theorem \ref{thm-1}, for any $x\in M$ and any nonzero tangent vector $u\in T_xM$, there is a $\delta(x)$-Killing vector field $X$ of $(M,F)$ for the tangent vector $u$.
Denote $g_X^F$ the localization of $F$ at $X$ defined by the Hessian matrices
evaluated with the base vector $X(\cdot)$ at each point. This Riemannian metric is well defined in a neighborhood $\mathcal{U}$ of $x$ where $X$ is non-vanishing
at each point. Because $X$ is a Killing vector field of $(M,F)$, it is
also a Killing vector field of $(\mathcal{U},g_X^F)$. By Lemma 3.1 in
\cite{DX2014}, the integration
curve of $X$ passing $x$ is a geodesic of $(M,F)$. Applying Proposition \ref{prop-1},
we have $K^F(x,u,u\wedge v)=K^{g_X^F}(x,u\wedge v)$, where $v\in T_xM$ is linearly independent of $u$.
Because the length function of $X$ for the metric $F$ coincides
with that for the metric $g_X^F$ inside $\mathcal{U}$, i.e.
$$F(X(\cdot))^2\equiv\langle X(\cdot),X(\cdot)\rangle^F_{X(\cdot)} \mbox{ in }
\mathcal{U}.$$
So the length function of the Killing vector field $X$ for $(\mathcal{U},g_X^F)$ also achieves its maximum
at $x$, by \cite{Ber2} or Lemma 2.2 in \cite{Wa}, we get $K^{g_X^F}(x,u\wedge v)\geq 0$ for any $v\in T_xM$ which is linear independent with $u$. So $K^F(x,u,u\wedge v)\geq 0$, i.e.
$(M,F)$ is non-negatively curved.
\end{proof}

\begin{proposition} \label{prop-2}
Assume $(G/H,F)$ is a $G$-$\delta$-homogeneous Finsler
metric, in which $G$ is a connected Lie group and $H$ is the compact isotropy subgroup at $o=eH\in G/H$.
Then we have the following:
\begin{description}
\item{\rm (1)} Denote $H_0$ the identity component of $H$. Then the metric $F$ naturally induces a $G$-$\delta$-homogeneous
metric on $G/H_0$, still denoted as $F$, such that the canonical projection
$\pi:G/H_0\rightarrow G/H$ is locally isometric.
\item{\rm (2)} For any closed subgroup $K$ of $G$ containing $H_0$, we have a $G$-$\delta$-homogeneous Finsler metric $F'$ on $G/K$
induced by $F$ on $G/H_0$ such that the canonical projection map
$\pi:G/H_0\rightarrow G/K$ is a Finsler  submersion.
\end{description}
\end{proposition}

\begin{proof} The proof for (1) is very easy.
We only need to prove (2) with $H_0=H$.
We have an $\mathrm{Ad}(K)$-invariant decomposition
$\mathfrak{g}=\mathfrak{k}+\mathfrak{p}$ and $\mathrm{Ad}(H)$-invariant
decomposition $\mathfrak{k}=\mathfrak{h}+\mathfrak{m}'$ such that
$\mathfrak{m}=\mathfrak{m}'+\mathfrak{p}$, with projections
$\mathrm{pr}_\mathfrak{p}$ and $\mathrm{pr}_\mathfrak{m}$ accordingly.

Denote $\mathcal{C}$ the union of all the $\mathrm{Ad}(G)$-orbits of $\delta$-vectors, for all vectors $u$ in $\mathfrak{m}$ with $F(u)\leq 1$.
Then the $G$-$\delta$-homogeneous Finsler metric $F$
can be uniquely determined by
\begin{equation}\label{0004}
\mathrm{pr}_\mathfrak{m}(\mathcal{C})=\{u\in\mathfrak{m}\mbox{ with }F(u)\leq 1\}.
\end{equation}
Let $F'$ be Minkowski norm on $\mathfrak{p}$ induced by
the Minkowski norm $F$ on $\mathfrak{m}$ and the Finsler submersion $\mathrm{pr}_\mathfrak{p}|_\mathfrak{m}:\mathfrak{m}\rightarrow\mathfrak{p}$. Then we also have
\begin{equation}\label{0005}
\mathrm{pr}_\mathfrak{p}(\mathcal{C})=\{u\in\mathfrak{p}\mbox{ with }F'(u)\leq 1\}.
\end{equation}
By similar arguments as Lemma 3.1 in \cite{XD2016}, and all the essential conditions: the $\mathrm{Ad}(G)$-invariance of $\mathcal{C}$, and the similarity among (\ref{0004}), (\ref{0005}) and (\ref{0001}), we see $\mathcal{C}$
defines a $G$-homogeneous Finsler metric $F'$ on $G/K$, and using the translations of $G$, the canonical projection map $\pi:(G/H,F)\rightarrow
(G/K,F')$ is a Finsler submersion. Finally, we can see that for each nonzero $v\in\mathfrak{p}$, it has a horizonal lift $\tilde{v}\in\mathfrak{m}$. Then
the $\delta$-vector for $\tilde{v}$ with respect to $G/H$, is also
a $\delta$-vector for $v$ with respect to $G/K$. Thus 
$(G/K,F')$ is $G$-$\delta$-homogeneous.
\end{proof}

Notice the closed subgroups $H$ and $K$ in Proposition \ref{prop-2} is not required
to be compact, so the $G$-action on $G/K$ may not be effective. But it will
not affect the $\delta$-homogeneity here or later discussions.
\subsection{Quasi-compactness and the submersion construction}

When the connected Lie group $G$ is a quasi-compact Lie group, i.e.
$\mathfrak{g}$ is a compact Lie algebra,
or equivalently, the universal cover of $G$ is a product of compact semisimple Lie group and an Abelian group $\mathbb{R}^k$,
then a $G$-homogeneous Finsler space $(G/H,F)$ is called $G$-normal homogeneous, when $F$ is induced
by a bi-invariant Finsler metric $\bar{F}$ on $G$ such that the projection
$\pi:G\rightarrow G/H$ is a Finsler submersion \cite{XD2016}. It is not hard to see all $G$-normal homogeneous Finsler
metrics are $G$-$\delta$-homogeneous as well.

On the other hand,
consider a $G$-$\delta$-homogeneous Finsler space $(G/H,F)$ with a connected
quasi-compact $G$.
We may choose an $\mathrm{Ad}(G)$-invariant inner product on $\mathfrak{g}$, with respect to
which the decomposition $\mathfrak{g}=\mathfrak{h}+\mathfrak{m}$ is orthogonal. In this context,
we have the following lemma.
\begin{lemma} Assume $(G/H,F)$ is a $G$-$\delta$-homogeneous
Finsler space with the connected quasi-compact Lie group $G$ acting effectively on $G/H$. Then
\begin{equation}\label{0002}
\tilde{F}(w)=\max_{g\in G} F(\mathrm{pr}_\mathfrak{m}(\mathrm{Ad}(g)w))
\end{equation}
defines an $\mathrm{Ad}(G)$-invariant singular norm on $\mathfrak{g}$ and correspondingly a bi-invariant singular metric
on $G$, such that the canonical projection $\pi:(G,\tilde{F})\rightarrow
(G/H,F)$ is a submersion.
\end{lemma}

\begin{proof}
Firstly we observe $\tilde{F}$ is well defined when $G$ is quasi-compact, because only the compact semisimple factor is relevant for defining $\tilde{F}$. Secondly, we need to prove $\tilde{F}$ satisfies all the
conditions for singular norms. The argument is standard and easy.
Thirdly, it is obvious to see the singular norm $\tilde{F}$ on $\mathfrak{g}$ is $\mathrm{Ad}(G)$-invariant, thus it defines a bi-invariant
singular metric on $G$.
Finally, we prove the canonical projection $\pi:(G,\tilde{F})\rightarrow (G/H,F)$ is a Finsler submersion. The tangent map
$\pi_*:(\mathfrak{g},\tilde{F})\rightarrow (\mathfrak{m},F)$ coincides with
$\mathrm{pr}_\mathfrak{m}$. It is a Finsler submersion by Theorem \ref{thm-1}.
By the bi-invariance of $\tilde{F}$, the $G$-homogeneity of $F$, and
the argument for proving Lemma 3.1 in \cite{XD2016}, the tangent map for the projection
$\pi:(G,\tilde{F})\rightarrow (G/H,F)$ is a submersion everywhere.
\end{proof}

According to \cite{BN2009}, we call singular norm $\tilde{F}$ defined by (\ref{0002})
the {\it Chebyshev norm}, and the corresponding bi-invariant singular metric the {\it Chebyshev metric}. To summarize, when
$G$ is quasi-compact, a $G$-$\delta$-homogeneous Finsler metrics can be
induced by a bi-invariant singular metric and the submersion.
It is easy to check any Finsler metric produced in this process on $G/H$,
whenever it is smooth, must be a $G$-$\delta$-homogeneous Finsler metric.
The horizonal lifts may not be unique, but they provides the $\delta(x)$-Killing vector fields for all $x\in G/H$ and all tangent vectors
in $T_x(G/H)$.
V. N. Berestovskii and
Yu. G. Nikonorov have already made these observations when $G$ is compact
or $G/H$ is compact \cite{BN2009}. Now the problem is

\begin{question}
Do we still have this construction when the connected $\delta$-homogeneous Finsler space $(M,F)$ is not compact?
\end{question}

The answer is positive. We just need to choose a suitable closed connected quasi-compact subgroup $G\subset I_0(M,F)$ which acts transitively and effectively on $M$.

\begin{theorem}\label{quasi-compactness-thm}
Assume $(M,F)$ is a connected $\delta$-homogeneous Finsler space. Let $G$ be the smallest closed connected subgroup of $I_0(M,F)$ which Lie algebra
$\mathfrak{g}$ contains all the $\delta(o)$-Killing vector fields in $\mathrm{Lie}(I_0(M,F))$ for a fixed $o\in M$ and all tangent vectors $u\in T_oM$, then $G$ is quasi-compact
and $(M,F)$ is $G$-$\delta$-homogeneous.
\end{theorem}

The group $G$ in Theorem \ref{quasi-compactness-thm} can be constructed as following. First, we fix $o\in M$ and use all the $\delta(o)$-Killing vectors (or all the $\delta$-vectors) in $\mathrm{Lie}(I_0(M,F))$
to generate a subalgebra $\mathfrak{g}'$, which corresponds to a connected subgroup $G'\subset I_0(M,F)$. Then $G$ is the closure of $G'$ in $I_0(M,F)$.
To prove Theorem \ref{quasi-compactness-thm}, we will need the following lemma.

\begin{lemma} \label{lemma-5}
Let $G'\subset I_0(M,F)$ be the connected Lie group generated by all $\delta$-vectors for a $\delta$-homogeneous Finsler space $(M,F)$, and $G$ be the closure of $G'$ in $I_0(M,F)$. Denote their Lie algebras as $\mathfrak{g}'$ and $\mathfrak{g}$ respectively. Then we have the following:
\begin{description}
\item{\rm (1)}If $\mathfrak{g}''\subset\mathfrak{g}$ is a real linear subspace
satisfying $[\mathfrak{g}'',\tilde{u}]\subset\mathfrak{g}''$ for all
$\delta$-vectors $\tilde{u}\in\mathfrak{g}'$, then $\mathfrak{g}''$ is
an ideal of $\mathfrak{g}$.
\item{\rm (2)}If $\mathfrak{g}''\subset\mathfrak{g}$ is a real linear subspace
satisfying $[\mathfrak{g}'',\tilde{u}]=0$ for all
$\delta$-vectors $\tilde{u}\in\mathfrak{g}'$, then $\mathfrak{g}''$ is
contained in the center of $\mathfrak{g}$.
\end{description}
\end{lemma}

\begin{proof} Because $[\mathfrak{g}'',\tilde{u}]\subset\mathfrak{g}''$ for all $\delta$-vectors $\tilde{u}$ which generate $\mathfrak{g}'$,
we have $\mathrm{Ad}(G')\mathfrak{g}''\subset\mathfrak{g}''$. It is still
valid when we replace $G'$ with its closure $G$ in $I_0(M,F)$, i.e.
$\mathrm{Ad}(G)\mathfrak{g}''\subset\mathfrak{g}''$. So $\mathfrak{g}''$ is
an ideal of $\mathfrak{g}$. The proof for (1) is done.
The proof for (2) is similar.
\end{proof}

{\bf Proof of Theorem \ref{quasi-compactness-thm}.}
Because $\mathfrak{g}=\mathrm{Lie}(G)$ contains all $\delta(o)$-Killing vector fields,
$G$ acts transitively around a neighborhood of $o$, and thus transitively
everywhere. By Theorem \ref{thm-1}, $(M,F)$ is $G$-$\delta$-homogeneous.
Denote $M=G/H$ with an $\mathrm{Ad}(H)$-invariant decomposition
$\mathfrak{g}=\mathfrak{h}+\mathfrak{m}$.

We only need to prove $\mathfrak{g}$ is compact.
We take Levi decomposition $\mathfrak{g}=\mathfrak{g}_1+\mathfrak{g}_2$ where
$\mathfrak{g}_1$ is a semisimple subalgebra, and $\mathfrak{g}_2$ is a solvable
ideal.
If we can prove $\mathfrak{g}_1$ is compact, and $\mathfrak{g}_2$
is the center of $\mathfrak{g}$, then the compactness of $\mathfrak{g}$ is done.


First we prove $\mathfrak{g}_2$ is the center of $\mathfrak{g}$. Consider any $\delta$-vector $\tilde{u}$ in $\mathfrak{g}$ and $v\in[\mathfrak{g}_2,\mathfrak{g}_2]$.
Because $[\mathfrak{g}_2,\mathfrak{g}_2]$ is a nilpotent ideal of $\mathfrak{g}$, the right side of
\begin{eqnarray}\label{0003}
\mathrm{pr}_\mathfrak{m}(\mathrm{Ad}(\exp tv)\tilde{u})
=\mathrm{pr}_\mathfrak{m}(\tilde{u})+t\mathrm{pr}_\mathfrak{m}([v,\tilde{u}])+
\frac{1}{2}t^2\mathrm{pr}_\mathfrak{m}([v,[v,\tilde{u}]])+\cdots,
\end{eqnarray}
is a finite sum. Because $\tilde{u}$ is
a $\delta$-vector, the function $f(t)=F(\mathrm{pr}_\mathfrak{m}(\mathrm{Ad}(\exp tv)\tilde{u}))$ is bounded for $t\in\mathbb{R}$. So the vector coefficient in (\ref{0003}) for each positive power of $t$
must vanish. In particular, $\mathrm{pr}_\mathfrak{m}([v,\tilde{u}])=0$,
i.e. $[v,\tilde{u}]\in\mathfrak{h}$ for any $v\in[\mathfrak{g}_2,\mathfrak{g}_2]$ and $\delta$-vector $\tilde{u}$.
In fact, we have $$[[\mathfrak{g}_2,\mathfrak{g}_2],\tilde{u}]\subset
\mathfrak{h}\cap[\mathfrak{g}_2,\mathfrak{g}_2]$$
because $[\mathfrak{g}_2,\mathfrak{g}_2]$ is an ideal of $\mathfrak{g}$.
By (1) of Lemma \ref{lemma-5}, $\mathfrak{h}\cap[\mathfrak{g}_2,\mathfrak{g}_2]$ is an ideal of $\mathfrak{g}$ contained in $\mathfrak{h}$. It must be 0 because $G\subset I_0(M,F)$ acts effectively on $M$. By (2) of Lemma \ref{lemma-5},
$[\mathfrak{g}_2,\mathfrak{g}_2]$ is contained in the center of $\mathfrak{g}$.

Now we consider (\ref{0003}) for any $v\in\mathfrak{g}_2$ and $\delta$-vector $\tilde{u}$.
Because $[\mathfrak{g}_2,\mathfrak{g}_2]$ is contained in the center of
$\mathfrak{g}$, the right side of (\ref{0003}) is a finite sum. Similar arguments as above proves $\mathfrak{g}_2$ is contained in the center of
$\mathfrak{g}$.
Because $\mathfrak{g}_1$ is semi-simple, this proves $\mathfrak{g}_2$ is the center of $\mathfrak{g}$.

Then we prove $\mathfrak{g}_1$ must be compact. Assume conversely it is not. Let $G_1$ be the semisimple
Lie group corresponding to $\mathfrak{g}_1$. We can find a maximal compact subalgebra $\mathfrak{k}\subset\mathfrak{g}_1$ such that $\mathfrak{h}\subset
\mathfrak{h}'=\mathfrak{k}+\mathfrak{g}_2$. Denote $H'$ the closed subgroup
of $G$ corresponding to $\mathfrak{h}'$. By Proposition \ref{prop-2} and Theorem \ref{non-nega-thm}, $G/H'$ admits a $G$-$\delta$-homogeneous Finsler
metric $F''$, which is a non-negatively curved.
By the Iwasawa decomposition $G_1=NAK$, where $A$ is Abelian and $N$ is nilpotent, the metric $F'$ on $G/H'$ is a left invariant Finsler metric on
the solvable Lie group
$G''=NA$. Because we have assumed that $\mathfrak{g}_1$ is not compact, both $N$ and $A$ has positive dimensions, i.e. $\dim G''\geq 2$.
Denote $\mathfrak{g}''=\mathfrak{a}+\mathfrak{n}$,
then $[\mathfrak{g}'',\mathfrak{g}'']=\mathfrak{n}$. We can find a nonzero vector $u\in\mathfrak{g}''\backslash\mathfrak{n}$, such that $\langle u,\mathfrak{n}\rangle_{u}^{F''}=0$. By Lemma 4.3 in \cite{Huang},
$\mathrm{Ric}^{F''}(u)\leq 0$ with equality only happens when $\mathrm{ad}(u)$ is skew symmetric with respect to $g_u^{F''}$. Because $(G/K,F'')$ has non-negative
flag curvature, so $\mathrm{Ric}^{F''}(u)=0$, and $\mathrm{ad}(u)$ is skew symmetric with respect to $g_u^{F''}$. It implies the nonzero eigenvalues of $\mathrm{ad}u:\mathfrak{g}''\rightarrow\mathfrak{g}''$ are pure imaginary numbers. On the other hand $\mathrm{ad}(u)$ only has real eigenvalues, so $u$
must be a nilpotent element in $\mathfrak{n}$. This is a contradiction to our
choice of $u$ from $\mathfrak{g}''\backslash\mathfrak{n}$.

The proof of Theorem \ref{quasi-compactness-thm} is done.

All above discussions can be summarized as the following theorem.
\begin{theorem}\label{thm-3}
Let $(M,F)$ be a connected $\delta$-homogeneous Finsler metric, either compact or noncompact. Then we can find a closed connected quasi-compact subgroup $G\in I_0(M,F)$ such that $(M,F)$ is $G$-$\delta$-homogeneous, and $F$ can be
induced by the Chebyshev metric on $G$ (i.e. the singular bi-invariant Finsler metric $\tilde{F}$ on $G$ determined by $F$ and (\ref{0002})) and submersion.
\end{theorem}

\subsection{Normal homogeneity and $\delta$-homogeneity}

Theorem \ref{thm-3} implies all connected $\delta$-homogeneous Finsler metrics, no matter compact or noncompact, can be constructed by the same process as
normal homogeneous Finsler metrics. The only difference is that for
the suitably chosen connected quasi-compact group $G$, we use (smooth) bi-invariant Finsler metrics for $G$-normal homogeneity, but singular bi-invariant
Finsler metrics (i.e. the Chebyshev metrics) for $G$-$\delta$-homogeneity.

Notice for any connected $G$-$\delta$-homogeneous Finsler space $(G/H,F)$ where $G$ is connected and quasi-compact, we may have
many different bi-invariant singular metrics on $G$ which defines the same
$F$, among which the Chebyshev metric $\tilde{F}$ is the smallest one.
In particular, when $(G/H,F)$ is $G$-normal homogeneous, there exists
a smooth one, i.e. a bi-invariant Finsler metric $\bar{F}$, which meets our purpose. But generally speaking, these smooth $\bar{F}$ may not exist, i.e.
$G$-normal-homogeneity and $G$-$\delta$-homogeneity are essentially different.

Here is an example. Let $M$ be coset space $S^3=\mathrm{U}(2)/\mathrm{U}(1)$
where $\mathrm{U}(1)$ corresponds the right down corner. We denote each matrix
$$\left(
   \begin{array}{cc}
     \sqrt{-1}(a+d) & b+\sqrt{-1}c \\
     -b+\sqrt{-1}c & \sqrt{-1}(-a+d) \\
   \end{array}
 \right)\in\mathfrak{u}(2)
$$ as $(a,b,c,d)$. Then $\mathfrak{h}=\mathbb{R}(1,0,0,-1)$ and
$\mathfrak{m}=\{(a,b,c,a)\mbox{ for all }a,b,c\in\mathbb{R}\}$.
We have a family of $\mathrm{U}(2)$-homogeneous Riemannian metrics $F_\epsilon$, defined by
$$F_\epsilon(a,b,c,a)=(4-\epsilon)a^2+b^2+c^2$$ on $\mathfrak{m}$.
When $\epsilon$ is a positive number, sufficiently close to 0,
$F_\epsilon$ is an $\mathrm{U}(2)$-normal homogeneous Riemannian metric.
Their limit when $\epsilon$ approaches 0, i.e. $F_0$, is
$\mathrm{U}(2)$-$\delta$ homogeneous (see Lemma \ref{lemma-3} below).

Now we show $F_0$ can not be induced by a bi-invariant Finsler metric
$\bar{F}$ on $\mathrm{U}(2)$ and submersion.
Assume conversely such a smooth metric exists.
Then the indicatrix $\mathcal{I}_{\bar{F},\mathfrak{g}}$ in $\mathfrak{g}$
must be contained in the region
$$\{(a,b,c,d)|(a+d)^2+b^2+c^2\leq 1\}\subset\mathfrak{g}.$$
Because it is invariant
under right $\mathrm{SU}(2)$-translations, $\bar{F}$ must be an
$(\alpha,\beta)$-metric corresponding to the Lie algebra decomposition
$\mathfrak{u}=\mathfrak{su}(2)\oplus\mathbb{R}$. So $\mathcal{I}_{\bar{F},\mathfrak{g}}$ is contained in the region
$$\{(a,b,c,d)|\sqrt{a^2+b^2+c^2}\leq 1-|d|\}.$$
On the other hand, $\mathcal{I}_{\bar{F},\mathfrak{g}}$ must contain
the round sphere
$$\{(a,b,c,0)|a^2+b^2+c^2=1\}\subset\mathfrak{su}(2),$$
where $\mathcal{I}_{\bar{F},\mathfrak{g}}$ lost its smoothness. This provides
the contradiction.

\subsection{Proof of Theorem \ref{main-thm-1}}

It is a fundamental observation that singular norms and singular metrics can be approximated by smooth ones. For proving Theorem \ref{main-thm-1},
we will only
use the following lemma for singular norms.

\begin{lemma}\label{appro-lemma}
For any singular norm $F$ on an $n$-dimensional real vector space $\mathbf{V}$, we can find a sequence of smooth Minkowski norm $F_n$ on $\mathbf{V}$ which converge to $F$ in the local
$C^0$-topology.
\end{lemma}
\begin{proof}
Denote $|\cdot|$ and $d\mathrm{vol}_x$ the standard Euclidean norm and volume form on $\mathbf{V}$ respectively.
We can find a family of smooth non-negative functions $\psi_\epsilon(x)$ on $\mathbb{R}^n$,
with the parameter $\epsilon>0$,
such that each $\psi_\epsilon(x)$ is supported in $\mathcal{B}_\epsilon=
\{x\in\mathbf{V}\mbox{ with }|x|\leq\epsilon\}$, and
$\int_{x\in\mathbf{V}}\psi_\epsilon(x)dx=1$. the convolution between $F$ and each $\psi_\epsilon$ defines a family of smooth functions
$$F_{1;\epsilon}(x)=
\int_{y\in\mathbf{V}}F(x-y)\psi_\epsilon(y)d\mathrm{vol}_y$$
on $\mathbb{R}^n$. When the positive parameter $\epsilon$ approaches 0,
$F_{1;\epsilon}$ is locally $C^0$-convergent to $F$.
Because $F$ is convex, i.e. for any $x_1,x_2\in\mathbf{V}$ and $\lambda\in[0,1]$,
$F(\lambda x_1+(1-\lambda)x_2)\leq\lambda F(x_1)+(1-\lambda)F(x_2)$,
we also have
\begin{eqnarray*}
F_{1;\epsilon}(\lambda x_1+(1-\lambda)x_2)&=&
\int_{y\in\mathbf{V}}F(\lambda x_1+(1-\lambda)x_2-y)\psi_\epsilon(y)d\mathrm{vol}_y\\
&\leq&\int_{y\in\mathbf{V}}(\lambda F(x_1-y)
+(1-\lambda)F(x_2-y))\psi_\epsilon(y)d\mathrm{vol}_y\\
&=&\lambda F_{1;\epsilon}(x_1)+(1-\lambda) F_{1;\epsilon}(x_2).
\end{eqnarray*}
So for each $\epsilon>0$, $F_{1;\epsilon}$ is convex as well.

Now the problem is $F_{1;\epsilon}$ is not positively homogeneous in general. Suppose the positive $\epsilon$ is sufficiently closed to 0, then
the derivative of $F_{1;\epsilon}$ in the radius direction is non-vanishing everywhere in the region
$\{x\in\mathbb{R}^n\mbox{ with }1/2\leq |x|\leq 2\}$. The pre-image $\mathcal{S}_\epsilon=F_{1;\epsilon}^{-1}(1)$ is a smooth $n-1$-dimensional sphere surrounding the origin. Because $F_{1;\epsilon}$ is convex, $\mathcal{S}_\epsilon$ bounds a convex region in $\mathbf{V}$.
We define the second family of perturbation functions $F_{2;\epsilon}$ on $\mathbf{V}$, such that each $F_{2;\epsilon}$ is positively homogeneous
of degree one, and $F_{2;\epsilon}(x)=1$ iff $x\in\mathcal{S}_\epsilon$.
Then $F_{2;\epsilon}$ is a positive smooth function on
$\mathbb{R}^n\backslash\{0\}$. The convexity of $\mathcal{S}_\epsilon$ implies the Hessian $(g_{ij}(y))=(\frac12[F_{2;\epsilon}^2(y)]_{y^iy^j})$ is semi positive definite at least.

At last we define $F_{\epsilon}=\sqrt{F_{2;\epsilon}^2+\epsilon |x|^2}$,
then $F_{\epsilon}$ satisfies all the conditions for smooth Minkowski norms, when the positive parameter $\epsilon$ is sufficiently close to 0.
When $\epsilon$ approaches 0, $F_\epsilon$ uniformly converges to $F$ in
each compact subset in $\mathbf{V}$, i.e. in the local $C^0$-topology.
\end{proof}

We will also need the following simple technical facts.

\begin{lemma}\label{lemma-3}
Assume $G$ is a connected quasi-compact Lie group which acts transitively on the smooth coset space $G/H$. Let $F_n$ be a sequence of $G$-$\delta$-homogeneous Finsler
metrics on $G/H$ induced by submersion and the bi-invariant singular metrics $\bar{F}$ respectively. Denote $\tilde{F}_n$ the Chebyshev metric defined $F_n$ respectively. Then we have the following.
\begin{description}
\item{\rm (1)} If $\bar{F}_n$ converges to a singular metric $\bar{F}$ on $G$ in the local $C^0$-topology, then $F_n$ converges
    to a singular metric $F$ on $G/H$ in the local $C^0$-topology.
    Further more, $\bar{F}$ is bi-invariant, $F$ is $G$-homogeneous,
    and the canonical projection $\pi:(G,\bar{F})\rightarrow (G/H,F)$ is
    a submersion.
\item{\rm (2)} If $F_n$ converges to a singular metric $F$ on $G/H$
in the local $C^0$-topology, then $F$ can be induced by its Chebyshev
metric on $G$ and submersion. If $F$ is Finsler metric, then it must
be a $G$-$\delta$-homogeneous Finsler metric.
\end{description}
\end{lemma}

\begin{proof}
(1)
Assume $\bar{F}_n$ converges to the singular metric $\bar{F}$ on
$G$ in the local $C^0$-topology. Then $\bar{F}$ is bi-invariant because each
$\tilde{F}_n$ is. Similar arguments as for Lemma 3.1 in \cite{XD2016}, $\bar{F}$ defines a $G$-homogeneous singular metric $F$ on $G/H$. To show $F_n$
locally $C^0$-converges to $F$, we only need to prove it in $\mathfrak{m}$, viewing all $\tilde{F}_n$, $F_n$, $\tilde{F}$, and $F$ as singular norms. For each sufficiently small $\epsilon>0$, we can find $N>0$, whenever $n>N$, $|\bar{F}_n(\tilde{w})-\bar{F}(\tilde{w})|<\epsilon$
when $\bar{F}(\tilde{w})\leq 1$. It implies the following inclusions among subsets of $\mathfrak{g}$ for each $n>N$,
\begin{eqnarray*}
\{\tilde{w}\in\mathfrak{g}\mbox{ with }\bar{F}_n(\tilde{w})\leq\frac{1}{1+\epsilon}\}&\subset&
\{\tilde{w}\in\mathfrak{g}\mbox{ with }\bar{F}(\tilde{w})\leq 1\}\\
&\subset&
\{\tilde{w}\in\mathfrak{g}\mbox{ with }\bar{F}_n(\tilde{w})\leq\frac{1}{1-\epsilon}\}.
\end{eqnarray*}
Because of submersion, $\mathrm{pr}_\mathfrak{m}$ maps these subsets onto
the following subsets in $\mathfrak{m}$ with corresponding inclusions,
\begin{eqnarray*}
\{w\in\mathfrak{m}\mbox{ with }F_n(w)\leq\frac{1}{1+\epsilon}\}&\subset&
\{w\in\mathfrak{m}\mbox{ with }F(w)\leq 1\}
\\ &\subset&\{w\in\mathfrak{m}\mbox{ with }F_n(w)\leq\frac{1}{1-\epsilon}\}.
\end{eqnarray*}
So we also have $|F_n(w)-F(w)|\leq \epsilon$ when $n>N$ and $F(w)=1$. So as
singular norms on $\mathfrak{m}$, $F_n$ converges to $F$ in the local $C^0$-topology. Using the $G$-translations, this convergence is still valid
when all $F_n$ and $F$ are viewed as singular metrics.

(2) Viewed as singular norms on $\mathfrak{m}$, $F_n$ converges to $F$
in the local $C^0$-topology. Then the Chebyshev norms $\tilde{F}_n$
also converge to the Chebyshev norm $\tilde{F}$ defined by $F$. Because
$\mathrm{pr}_\mathfrak{m}:(\mathfrak{g},\tilde{F}_n)\rightarrow (\mathfrak{m},F_n)$ is a
submersion for each $n$, i.e.
$$\mathrm{pr}_\mathfrak{m}(\{\tilde{u}\in\mathfrak{g}\mbox{ with }\tilde{F}_n(\tilde{u})\leq 1\})=\{u\in\mathfrak{m}\mbox{ with }
F_n(u)\leq 1\},$$
passing to their limits, we also have
$$\mathrm{pr}_\mathfrak{m}(\{\tilde{u}\in\mathfrak{g}\mbox{ with }\tilde{F}(\tilde{u})\leq 1\})=\{u\in\mathfrak{m}\mbox{ with }
F(u)\leq 1\}.$$
That proves $F$ can be induced by its Chebyshev metric on $G$ and submersion.
The last assertion is an obvious observation we have made.
\end{proof}

Now we prove Theorem \ref{main-thm-1}.

We only need to prove two assertions for this theorem. The first one is that,
inside the space of all smooth Finsler metrics on $G/H$,
the subset of all $G$-$\delta$-homogeneous Finsler metrics
is closed in the local $C^0$-topology. This is already done by (2) of
Lemma \ref{lemma-3}.

The second one is that
for any $G$-$\delta$-homogeneous Finsler metric $F$, we can find a
sequence of $G$-normal homogeneous metrics $F_n$, such that $F_n$
converges to $F$ in the $C^0$-topology.

now we prove the second claim. Consider a
$G$-$\delta$-homogeneous Finsler metric
$F$ on $G/H$, which is defined by submersion from the Chebyshev metric $\tilde{F}$ on $G$. By Lemma \ref{appro-lemma},
we can find a sequence
of Minkowski norms $\{\tilde{F}'_n\}$ on $G$ which converge to
$\tilde{F}$ in the $C^0$-topology. Average each $\tilde{F}'_n$ with all $\mathrm{Ad}(G)$-actions as following,
$$\tilde{F}_n(w)=\sqrt{\frac{\int_{g\in G}{\tilde{F}'_n}(w)^2 d\mathrm{vol}_g}{\int_{g\in G}d\mathrm{vol}_g}},$$
where $\mathrm{vol}_g$ is a bi-invariant volume form of $G$,
we get a sequence of $\mathrm{Ad}(G)$-invariant Minkowski norms
$\{\tilde{F}_n\}$ on $\mathfrak{g}$, such that the
bi-invariant Finsler metrics $\tilde{F}_n$ converge to the bi-invariant singular metric $\tilde{F}$. Each bi-invariant smooth Finsler metric $\tilde{F}_n$ on $G$ defines a $G$-normal homogeneous metric $F_n$. By (1) of Lemma \ref{lemma-3}, $F_n$ converges to $F$ in the local $C^0$-topology.

This ends the proof of Theorem \ref{main-thm-1}.

\section{The positive curvature problem for $\delta$-homogeneous Finsler spaces}
Let $(G/H,F)$ be a $G$-$\delta$-homogeneous Finsler space, where $G$ is
a connected quasi-compact Lie group which acts effectively on $G/H$. We
choose an $\mathrm{Ad}(G)$-invariant inner product, and denote
$\mathfrak{g}=\mathfrak{h}+\mathfrak{m}$ the corresponding orthogonal
decomposition. Then according to
\cite{XD2016}, a subalgebra $\mathfrak{s}\subset\mathfrak{g}$ is called
a {\it flat splitting subalgebra} (FSS in short), when it satisfies the following
conditions:
\begin{description}
\item{\rm (1)} $\mathfrak{s}$ is the intersection of a family of
Cartan subalgebras of $\mathfrak{g}$.
\item{\rm (2)} $\mathfrak{s}=\mathfrak{s}\cap\mathfrak{h}+
\mathfrak{s}\cap\mathfrak{m}$.
\item{\rm (3)} $\dim\mathfrak{s}\cap\mathfrak{m}\geq 2$.
\end{description}

Because of Condition (1) for a FSS, $\mathfrak{s}$ generates
a closed connected abelian subgroup $\exp\mathfrak{s}$ of $G$. The orbit $\exp\mathfrak{s}\cdot o=\exp(\mathfrak{s}\cap\mathfrak{m})\cdot o$
can be viewed as the connected abelian Lie group $S=\exp(\mathfrak{s}\cap\mathfrak{m})$. The subspace metric $F|_{S\cdot o}$
is a left invariant Finsler metric on $S$, which is obviously flat.
The crucial observation here is that $(S\cdot o,F|_{S\cdot o})$
is totally geodesic.

\begin{lemma} \label{lemma-4}
Keep all above notations and assumptions. Then  $(S\cdot o,F|_{S\cdot o})$ is a totally geodesic subspace of $(G/H,F)$.
\end{lemma}

\begin{proof} By Theorem \ref{main-thm-1}, we assume $F$ is the limit
of a sequence $F_n$ of $G$-normal homogeneous metrics on $G/H$ in the
local $C^0$-topology. Denote $\tilde{F}$ and $\tilde{F}_n$ their Chebyshev norms on $\mathfrak{g}$ respectively. For each $n$, we have another bi-invariant Minkowski norm $\bar{F}_n$ defining the $G$-normal homogeneity of $F_n$. Consider any nonzero vector $u\in\mathfrak{s}\cap\mathfrak{m}$. By Lemma 3.4
in \cite{XD2016}, there exists a unique vector $\tilde{u}_n\in\mathfrak{s}$
such that $\mathrm{pr}_\mathfrak{m}(\tilde{u}_n)=u$ and
$\bar{F}_n(\tilde{u}_n)=\tilde{F}_n(\tilde{u}_n)=1$. Because $\tilde{F}_n$ converges to $\tilde{F}$ in the local $C^0$-topology, $\tilde{F}(\tilde{u}_n)$ is bounded. We can find a convergent subsequence
converging to some $\tilde{u}\in\mathfrak{s}$ with
$\mathrm{pr}_\mathfrak{m}(\tilde{u})=u$ and
$\tilde{F}(\tilde{u})=1$. So $\tilde{u}$ defines a $\delta(o)$-Killing vector
field $X$ of $(G/H,F)$ for the tangent vector $u\in\mathfrak{m}$. The integration curve of $X$ passing $o$ is then
a geodesic of $(G/H,F)$. On the other hand, it coincides with $\exp (t\tilde{u})\cdot o=\exp(tu)\cdot o$ which is a geodesic of $(S\cdot o,F|_{S\cdot o})$. By left $S$-translations, this proves $(S\cdot o,F|_{S\cdot o})$ is totally geodesic.
\end{proof}

The observation that a FSS is a totally geodesic flat subspace for a normal
homogeneous Finsler space is the key technique that we reduce the classification for positively curved normal homogeneous Finsler spaces to
a totally algebraic problem.

To be precise, in \cite{XD2016}, Theorem 3.3, the preparation lemmas, and the case by case discussion from Section 4 to Section 6, proves the following.

If $G/H$ does not admit positively curved $G$-normal homogeneous
Riemannian metrics, i.e. there exists a linearly independent commuting pair
of vectors in $\mathfrak{m}$, then we can find a closed subgroup $K$ in $G$
with $\mathfrak{h}\subset\mathfrak{k}\subset\mathfrak{g}$, with the corresponding orthogonal decomposition $\mathfrak{g}=\mathfrak{k}+\mathfrak{p}$, such that there exists a FSS
for $G/K$. Then we see $G/H$ does not admit positively curved $G$-normal homogeneous
Finsler metrics either. In particular, when $G/H$ is even dimensional,
we can choose $K=H$. In most cases, the FSS can be found among Cartan
subalgebras, which will be called a {\it flat splitting Cartan subalgebra} or
FSCS in short.

Because
of Theorem \ref{main-thm-1} and Lemma \ref{lemma-4}, this theory can be
applied to prove Theorem \ref{main-thm-2}, which implies the classification for positively curved $\delta$-homogeneous Finsler spaces coincides with that for positively curved normal homogeneous Riemannian spaces.

{\bf Proof of Theorem \ref{main-thm-2}.}
Let $(G/H,F)$ be a positively curved $G$-$\delta$-homogeneous Finsler spaces.
Assume conversely $G/H$ does not admit a positively curved $G$-normal homogeneous Riemannian metric, then we can find a closed subgroup $K$ of $G$,
such that there exists a FSS for $G/K$. By Theorem \ref{theorem-3-2} and Proposition \ref{prop-2},
$G/K$ also admits a positively curved $G$-$\delta$-homogeneous Finsler metric.
But by Lemma \ref{lemma-4} implies the FSS provides a totally geodesic
flat subspace of dimension bigger than 1. This is the contradiction.

Lemma \ref{lemma-4} can tell us more when we turn to the (FP) condition \cite{XD2016-2}.
\begin{definition} A Finsler space $(M,F)$ is called flag-wise positively curved or satisfying the (FP) condition, if for any $x\in M$, any tangent
plane $\mathbf{P}\subset T_xM$, we can find a $y\in\mathbf{P}$ such that
the flag curvature $K^F(x,y,\mathbf{P})>0$.
\end{definition}

The (FP) condition itself is very weak \cite{X-2016}. But its combination
with the non-negatively curved condition seems much stronger, very like
the positively curved condition. We guess a flag-wise positively curved and non-negatively curved homogeneous Finsler space $G/H$ must be compact.
Further more, when $G$ is compact, we guess $\mathrm{rank}G\leq\mathrm{rank}H+1$.

By Theorem \ref{non-nega-thm}, a flag-wise positively curved $\delta$-homogeneous Finsler space belongs to this special category.
We can prove
\begin{theorem} Let $(M,F)$ be a connected flag-wise positively curved $\delta$-homogeneous Finsler space which dimension is bigger than 1, then $M$ is compact. If $M=G/H$
where $G$ is a compact connected Lie group acting effectively on $M$,
then $\mathrm{rank}G\leq\mathrm{rank}H+1$.
\end{theorem}

\begin{proof} By Theorem \ref{quasi-compactness-thm},
we can assume the flag-wise positively curved $\delta$-homogeneous Finsler space $(M,F)$ is $G$-$\delta$-homogeneous, where $G$ is a connected, simply
connected and quasi-compact Lie group which acts transitively and effectively on $M$. We can present $G=G_1\times \mathbb{R}^k$ where
$G_1$ is a maximal compact subgroup of $G$, $M=G/H$ where $H$ is a closed
subgroup of $G_1$. First we prove then $\mathrm{rank}G\leq\mathrm{rank}H+1$.
If this is not true, any Cartan subalgebra of $\mathfrak{h}$ can be enlarged to a F.S.C.S in $\mathfrak{g}$. By Lemma \ref{lemma-4}, $(M,F)$ can not satisfy the (FP) condition because of the totally geodesic flat subspace.
This argument proves the second statement in the theorem. Now we prove the first statement, i.e. the compactness of $M$. Assume conversely $M$ is not
compact, then $k=1$ and $\mathrm{rank}G_1=\mathrm{rank}H$. Let $\mathfrak{t}$
be a Cartan subalgebra of $\mathfrak{h}$, with respect to which $\mathfrak{g}_1=\mathrm{Lie}(G_1)$
and $\mathfrak{h}$ can be decomposed as the sum of $\mathfrak{t}$ and root
planes $\mathfrak{g}_{\pm\alpha}$. Because $\mathrm{rank}G_1=\mathrm{rank}H$, Each root plane $\mathfrak{g}_{\pm\alpha}$ is contained in either $\mathfrak{h}$ or $\mathfrak{m}$. Because $\dim G/H>1$, we can find a root
plane $\mathfrak{g}_{\pm\alpha}$ in $\mathfrak{m}$. Let $u$ be any nonzero vector in $\mathfrak{g}_{\pm\alpha}$, $v$ any nonzero vector in the $\mathbb{R}$-factor, and $\mathrm{ker}\alpha$ the codimension one subspace of
$\mathfrak{t}$ where $\alpha$ takes zero values. Then $u$, $v$ and $\mathrm{ker}\alpha\subset\mathfrak{t}$ span a FSCS, which will be a contradiction to the (FP) condition.
\end{proof}

As the end, we remark that,
FSS can be viewed as an obstacle for normal or $\delta$-homogeneous Finsler metric either positively curved or satisfying the (FP) condition. For the positively curved condition, this obstacle can be passed downward by submersion, i.e. when $G/H$ admits a positively curved $G$-normal or $G$-$\delta$-homogeneous Finsler metric, then for any closed subgroup $K\subset G$ containing $H$, the naturally induced $G$-normal or $G$-$\delta$-homogeneous Finsler metric is also positively curved. But this
observation is not true for the (FP) condition, i.e. this obstacle can not
be passed downward for the (FP) condition. When $\dim G/H$ is even, $H$ itself
can be taken as $K$, so the classification work in \cite{XD2016} proves
\begin{corollary} \label{last-cor}
The following conditions are equivalent for an even dimensional smooth coset space $G/H$:
\begin{description}
\item{\rm (1)} It admits positively curved $G$-normal homogeneous Finsler metrics.
\item{\rm (2)} It admits positively curved $G$-$\delta$-homogeneous Finsler metrics.
\item{\rm (3)} It admits flag-wise positively curved $G$-normal homogeneous Finsler metrics.
\item{\rm (4)} It admits flag-wise positively curved $G$-$\delta$-homogeneous Finsler metrics.
\end{description}
\end{corollary}

Notice when $\dim G/H$ is odd, Corollary \ref{last-cor} is not true. Theorem 1.2 in \cite{XD2016-2} provides many examples of smooth coset spaces admitting flag-wise positively curved and non-negatively curved homogeneous Finsler metrics. By Corollary 1.4 in \cite{XD2016-2}, most of them do not admit positively curved homogeneous Finsler metrics. In each example
$(G/H,F)$ in Theorem 1.2 of \cite{XD2016-2}, $F$ is defined by the navigation with respect to a normal homogeneous Riemannian metric $F'$ and a nonzero Killing vector field $V$ with $F'(V(\cdot))\equiv\mathrm{const}<1$. Then for any $x\in G/H$ and any
$F'$-unit tangent vector $u'\in T_x(G/H)$, we have a $\delta(x)$-Killing vector
field $X$ of $(G/H,F')$ from $\mathfrak{g}$, for the tangent vector $u'$. Then $X+V$ is
a $\delta(x)$-Killing vector field of $(G/H,F)$ from $\mathfrak{g}\oplus\mathbb{R}$, for the $F$-unit tangent vector $u=u'+V(x)$,
which can exhaust all the tangent directions.
So $F$ is $G\times S^1$-$\delta$-homogeneous.

{\bf Acknowledgement.} The authors would like to thank Shaoqiang Deng and Fuhai Zhu for helpful discussions.

\end{document}